\documentclass[reqno]{amsart}

\newtheorem{theorem}{Theorem}
\newtheorem{lemma}[theorem]{Lemma}
\newtheorem{proposition}[theorem]{Proposition}

\theoremstyle{definition}
\newtheorem{definition}[theorem]{Definition}
\usepackage{graphicx}
\usepackage[all]{xy}
\usepackage{bbm,amssymb,amsthm,amsfonts,amsmath,pifont,stmaryrd}
\usepackage{lineno}
\usepackage{pstricks, enumerate, pst-node, pst-text, pst-plot}

\theoremstyle{remark}
\newtheorem{remark}[theorem]{\bf Remark}

\numberwithin{equation}{section}
\numberwithin{theorem}{section}

\newcommand{\intav}[1]{\mathchoice {\mathop{\vrule width 6pt height 3 pt depth  -2.5pt
\kern -8pt \intop}\nolimits_{\kern -6pt#1}} {\mathop{\vrule width
5pt height 3  pt depth -2.6pt \kern -6pt \intop}\nolimits_{#1}}
{\mathop{\vrule width 5pt height 3 pt depth -2.6pt \kern -6pt
\intop}\nolimits_{#1}} {\mathop{\vrule width 5pt height 3 pt depth
-2.6pt \kern -6pt \intop}\nolimits_{#1}}}

\newcommand{\intavl}[1]{\mathchoice {\mathop{\vrule width 6pt height 3 pt depth  -2.5pt
\kern -8pt \intop}\limits_{\kern -6pt#1}} {\mathop{\vrule width 5pt
height 3  pt depth -2.6pt \kern -6pt \intop}\nolimits_{#1}}
{\mathop{\vrule width 5pt height 3 pt depth -2.6pt \kern -6pt
\intop}\nolimits_{#1}} {\mathop{\vrule width 5pt height 3 pt depth
-2.6pt \kern -6pt \intop}\nolimits_{#1}}}

\newcommand{\mc}{\mathcal}

\newcommand{\F}{\mc{F}}
\newcommand{\G}{\mc{G}}

\newcommand{\R}{\mathbb{R}}
\newcommand{\T}{\mathbb{T}}
\newcommand{\N}{\mathbb{N}}
\newcommand{\Q}{\mathbb{Q}}
\newcommand{\Z}{\mathbb{Z}}

\newcommand{\al}{\alpha}
\newcommand{\be}{\beta}
\newcommand{\la}{\lambda}

\newcommand{\ind}{\mathbbm{1}}

\begin{document}

\title[Law of large numbers for certain cylinder flows]{Law of large numbers for certain cylinder flows}

\author[Patr\'icia Cirilo]{Patr\'icia Cirilo}
\address{Universidade Estadual Paulista, Rua Crist\'ov\~ao Colombo 2265, 15054-000, S\~ao Jos\'e do Rio Preto, Brasil.}
\email{prcirilo@ibilce.unesp.br}

\author[Yuri Lima]{Yuri Lima}
\address{Weizmann Institute of Science, Faculty of Mathematics and Computer Science, POB 26, 76100, Rehovot, Israel.}
\email{yuri.lima@weizmann.ac.il}

\author[Enrique Pujals]{Enrique Pujals}
\address{Instituto Nacional de Matem\'atica Pura e Aplicada, Estrada Dona Castorina 110, 22460-320, Rio de Janeiro, Brasil.}
\email{enrique@impa.br}


\date{\today}


\begin{abstract}
We construct new examples of cylinder flows, given by skew product extensions of irrational rotations
on the circle, that are ergodic and rationally ergodic along a subsequence of iterates. In particular,
they exhibit law of large numbers. This is accomplished by explicitly calculating, for a subsequence of
iterates, the number of visits to zero, and it is shown that such number has a gaussian distribution.
\end{abstract}

\maketitle


\section{Introduction}

The purpose of this paper is to construct examples of skew product extensions of irrational
rotations of the additive circle $\T=\R/\Z$ exhibiting {\it law of large numbers}. More specifically,
under some weak diophantine conditions on the irrational number $\al\in\R$, we construct
{\it roof functions} $\phi:\T\rightarrow\Z$ for which the skew product
\begin{equation}\label{intro - eq 1}
\begin{array}{rcrcl}
F&:&\T\times\Z&\longrightarrow &\T\times\Z\\
 & & (x,y)    &\longmapsto     &(x+\al,y+\phi(x))
\end{array}
\end{equation}
is ergodic and rationally ergodic along a subsequence of iterates. This, in particular,
implies that $F$ has a law of large numbers. See Subsection \ref{sub lln} for
the proper definitions.

One must, first of all, observe that $F$ has a natural invariant measure, given by the product
of the Lebesgue measure on $\T$ and the counting measure on $\Z$, and it is {\it infinite}.
In this situation, classical theorems of ergodic theory are not valid. For instance,
Birkhoff's averages converge to zero almost surely, and this leads us to the following question:
what would be a good candidate for a Birkhoff-type theorem in this context? Denoting by $S_n\psi$ the
Birkhoff sum of the $L^1$-function $\psi:\T\times\Z\rightarrow\R$, the most natural way is try to
find a sublinear sequence $(a_n)$ of positive real numbers and consider the averages $S_n\psi/a_n$.
However, by a result of J. Aaronson (see \S 2.4 of \cite{Aa4}),
there is never a universal sequence $(a_n)$ for which $S_n\psi/a_n$ converges pointwise to the right
value. Nevertheless, Hopf's theorem (see \S 2.2 of \cite{Aa4}) is an indication that some sort of
regularity might exist and it might still be possible, for a specific sequence $(a_n)$, that the
averages oscillate without converging to zero or infinity and so one can hope for a summability
method that smooths out the fluctuations and forces convergence. Such {\it second order ergodic
theorems} were considered by J. Aaronson, M. Denker, and A. Fisher in \cite{ADF}.

Another attempt of obtaining a Birkhoff-type theorem has been made by Aaronson, in
which he defined and constructed examples of {\it rationally ergodic} maps (see \S 3.3 of \cite{Aa4}).
These maps possess a sort
of C\`esaro-averaged version of convergence in measure: there is a sequence $(a_n)$ such that, for
every $L^1$-function $\psi$ and every subsequence $(n_k)$ of positive integers, there exists a further
subsequence $(n_l)$ such that $S_{n_l}\psi(x)/a_{n_l}$ converges C\`esaro almost surely to $\int\psi$.
This latter property is called
{\it weak homogeneity} and the sequence $(a_n)$ is called a {\it return sequence}. Weak homogeneity
implies the existence of law of large numbers (see Subsection \ref{sub lln}).

A natural program of investigation regards two kinds of questions.
\begin{enumerate}[(i)]
\item What are the conservative, ergodic, rationally ergodic maps?
\item What fluctuations can the Birkhoff sums have?
\end{enumerate}

Our goal in this work is to give contributions to these questions by constructing examples of the form
(\ref{intro - eq 1}) that are ergodic and rationally ergodic along a subsequence of iterates. Up to
our knowledge, the first examples of ergodic cylinder flows were given by A. Krygin \cite{Kr74} and
K. Schmidt \cite{Sc78}. Their examples differ in nature. Krygin assures the existence, for any
irrational $\al$, of a roof function $\phi$ for which $F$ is ergodic. Actually, there exist elegant
categorical proofs that the set of pairs $(\al,\phi)$, in various different contexts,  for which $F$
is ergodic forms a residual set. See \cite{BM92}, \cite{Ka03}. On the other hand, Schmidt constructed
an explicit example motivated by the theory of random walks. He considered $\al=(\sqrt 5-1)/4$ and the
roof function equal to the Haar function defined in Section \ref{section - construction of function},
which is actually the basis function for our example. Subsequent works \cite{CK76}, \cite{Co80},
\cite{Sc77} of J.-P. Conze and M. Keane and Schmidt himself extended the results to every irrational
$\al$ and also to the larger class of roof functions
\begin{equation}\label{intro - eq 2}
\phi(x)=(\beta+1)\cdot\ind_{\left[0,\frac{\beta}{\beta+1}\right)}(x)-\beta.
\end{equation}
There are many other works regarding this question. See for instance \cite{ALMN92}, \cite{FL06},
\cite{Fr00}, \cite{Or83}, \cite{Pa91}.

Regarding (ii), J. Aaronson and M. Keane further investigated Schmidt's example in \cite{AaK}.
They studied the asymptotic behavior of the number of visits to zero and proved that the
Birkhoff sums represent a sort of ``deterministic random walk''. In particular, they showed that if
$\al$ is quadratic surd\footnote{The irrational number $\al$ is {\it quadratic surd} if it satisfies
a quadratic equation with integer coefficients.} then $F$ is rationally ergodic with return sequence
$a_n\sim n/\sqrt{\log n}$.

Not much is known regarding rational ergodicity. There are actually few examples that have been
proved to be rationally ergodic. See for instance \cite{Aa2}, \cite{Aa3}, \cite{AaK}, \cite{AaS},
\cite{LS}, where this property is shown to hold in different contexts. With respect to cylinder
flows given by skew products extensions of irrational rotations on the circle, the only
known examples are those in \cite{AaK}.

The most significant contribution of our work is to construct a new class of cylinder flows that
are rationally ergodic along a subsequence of iterates and, in particular, possess law of large
numbers.

\begin{theorem}\label{thm 1}
For any $\al\in\R$ such that $\liminf_{q\rightarrow\infty}q\|q\al\|=0$, there exists a skew product
$$
\begin{array}{rcrcl}
F&:&\T\times\Z&\longrightarrow &\T\times\Z\\
 & & (x,y)    &\longmapsto     &(x+\al,y+\phi(x))
\end{array}
$$
such that
\begin{itemize}
\item[(a)] $\phi$ belongs to $L^p(\T)$, for every $p\ge 1$, and
\item[(b)] $F$ is ergodic.
\end{itemize}
If $\al$ is also divisible, then
\begin{itemize}
\item[(c)] $F$ is rationally ergodic along a subsequence of iterates.
In particular, it has a law of large numbers.
\end{itemize}
\end{theorem}

An irrational number $\al$ is {\it divisible} if it has a subsequence of continuants $(q_n)$ with a
certain divisibility property and such that $\lim_{n\rightarrow\infty}q_n\|q_n\al\|=0$.
See Subsection \ref{sub continued fractions} for the specific definitions. It is worth noting that
the set of $\al$ satisfying these two conditions has full Lebesgue measure, according to the content
of Appendix \ref{appendix continued fractions}. Thus, in contrast to \cite{AaK}, in which the set
of parameters is countable, Theorem \ref{thm 1} holds for a set of
parameters of full Lebesgue measure\footnote{In a previous version of this paper, Theorem \ref{thm 1}
required stronger conditions on $\al$ for which the set of parameters has zero Lebesgue measure,
but it was pointed to us that the proof works for any divisible irrational number.}.

A remarkable feature of Theorem \ref{thm 1} is that the number of visits to zero along the iterates
in which $F$ is rationally ergodic exhibits a gaussian distribution. The return sequence is given
by $a_{q_{n+1}}=q_{n+1}/\sqrt{\pi n}$ and the normalized averages, described in equation
(\ref{equation norm average}), do not depend on the choice of $\al$ neither on the sequence $(q_n)$.

The roof function we construct is different in nature from the others used in this context.
We consider the Haar function $T$ defined in Section \ref{section - construction of function}
as a basis function and let
$$\phi(x)=\dfrac{1}{2}\sum_{j\ge 1}\Big[T(q_j(x+\al))-T(q_jx)\Big]$$
for a specific chosen sequence of positive integers $(q_n)$. One can see $\phi$ as the limit of
worser and worser coboundaries
\begin{equation}\label{intro - eq 3}
\phi_n(x)=\dfrac{1}{2}\sum_{j=1}^n\Big[T(q_j(x+\al))-T(q_jx)\Big].
\end{equation}
Observe that, if we just consider the coboundary $\phi_n$, the  respective cylinder flow will not
be ergodic and, moreover, will be conjugate to a rigid rotation. The increasing bad feature of
each $\phi_n$ is what will guarantee that $\phi$ has the required properties. The sequence
$(q_n)$ will be chosen via the continued fraction expansion of $\al$ and this is why the diophantine
properties of $\al$ influence the dynamical properties of $F$. Even though $\phi$ is unbounded,
the good feature of it is that we can explicitly calculate the number of visits to zero along a
sequence of iterates of $F$. See Lemma \ref{proposition - returns for rational} and Subsection
\ref{sub returns for F}.

The paper is organized as follows. In Section \ref{section preliminaries} we introduce the basic
notations and definitions as well as the necessary background for the sequel. Section
\ref{section - construction of function} is devoted to the construction of the roof function $\phi$
and the related convergence issues. In Section \ref{section - ergodicity} we establish the ergodicity
of $F$ with the aide of the theory of random walks. To this matter, Appendix \ref{appendix random walk}
treats the required results, adapted to our context. Section \ref{section - number of returns}
calculates the number of returns of a generic point to its fiber, assuming that $\al$ is divisible.
This in particular implies the second part of Theorem \ref{thm 1}, which is the content of Section
\ref{section - proof of rat ergodicity}. In Appendix \ref{appendix continued fractions}, we enclose the
results on continued fractions that allows us to state our results in the greatest possible generality.

\begin{remark}
In some sense, our construction resembles Anosov-Katok method of fast approximations developed in
\cite{AnKa}. Indeed, the referred maps are obtained as limits of periodic maps and here we will also
use this perspective (see Subsection \ref{sub phi tilde}). Another example that resembles ours is
Hajian-Ito-Kakutani's map. See \S 3.3 of \cite{Sa10} for a detailed exposition of this map.
\end{remark}

\section{Preliminaries}\label{section preliminaries}

\subsection{General notation}\label{sub general notation}

Given a set $X$, $\#X$ denotes the cardinality of $X$. If $A$ is a subset of $X$, $\ind_A:X\rightarrow\{0,1\}$
denotes the characteristic function of $A$:
$$
\ind_A(x)=\left\{
\begin{array}{lcl}
1&,&\text{ if }x\in A\\
0&,&\text{ if }x\in X\backslash A.
\end{array}
\right.
$$
$\Z$ denotes the set of integers and $\N$ the set of positive integers.
Each $n\in\N$ defines the ring $\Z_n$ of the residue classes module $n$. A
{\it complete residue system} is a set $\{a_1,\ldots,a_n\}$ of integers such that
$\{a_1,\ldots,a_n\}$ modulo $n$ is equal to $\Z_n$.

Given a real number $x$, $\lfloor x\rfloor$ and $\{x\}$ are the integer and fractional
parts of $x$, respectively. Let $\|x\|$ be the distance from $x$ to the closest integer,
$$\|x\|=\min\{\{x\},1-\{x\}\}.$$

We use the following notation to compare the asymptotic of functions.

\begin{definition}\label{def vinogradov}
Let $f,g:\N\rightarrow\R$ be two real-valued functions. We say $f\lesssim g$ if there is a
constant $C>0$ such that
$$|f(n)|\le C\cdot |g(n)|\,,\ \ \forall\,n\in\N.$$
If $f\lesssim g$ and $g\lesssim f$, we write $f\sim g$. We say $f\approx g$ if
$$\lim_{n\rightarrow\infty}\dfrac{f(n)}{g(n)}=1.$$
\end{definition}

Let $\T=\R/\Z$ denote the circle, parameterized by $[0,1)$, and let $d:\T\times\T\rightarrow\R$ be the
induced distance function. For every $\al\in\R$, $R_\al:\T\rightarrow\T$
is the rotation $R_\al x=x+\al$.

Let $\la$ be the Lebesgue measure on $\T$ and $\mu$ the measure defined on the cylinder $\T\times\Z$
by $\mu=\la\times$ counting measure on $\Z$. Given a function $\psi:\T\rightarrow\R$, its $L^p$-norm
with respect to $\la$ is defined as
$$\|\psi\|_p=\left(\int_\T |\psi|^pd\la\right)^{1/p}$$
and the space of $L^p$-integrable functions as $L^p(\T)$. Due to the index $p$, there will be no confusion
between the integer norm $\|\cdot\|$ and the $L^p$-norm $\|\cdot\|_p$.

\subsection{Continued fractions}\label{sub continued fractions}

Given an irrational number $\al$, consider its {\it continued fraction expansion}
$$\al\ =\ a_0+\dfrac{1}{a_1+\dfrac{1}{a_2+\dfrac{\ \ \ 1\ \ \ }{\ddots}}}\ :=\ [a_0;a_1,a_2,\ldots]\ ,$$
whose $n^{th}$-{\it convergent} is
$$\al_n=\dfrac{p_n}{q_n}=[a_0;a_1,a_2,\ldots,a_n],\ n\ge 0.$$

The $q_n$ are called the {\it continuants}. They give the best rational
approximations to $\al$. More precisely, the approximation is equal to
$$\|q_n\al\|=q_n\cdot\left|\al-\dfrac{p_n}{q_n}\right|\,\cdot$$
It is known, by Dirichlet's theorem, that
$$\liminf_{q\rightarrow\infty}q\|q\al\|\le 1$$
for any $\al\in\R$. Let $\alpha$ be {\it divisible} if it has a sequence $(q_{n_j})$ of continuants
satisfying
$$2q_{n_j}\text{ divides }q_{n_{j+1}}\ \ \text{ and }\ \ \lim_{j\rightarrow\infty}q_{n_j}\|q_{n_j}\al\|=0.$$
The set of divisible numbers has full Lebesgue measure in $\R$. This is the content of
Proposition \ref{proposition cf 1} which, in particular, guarantees that Theorem \ref{thm 1}
is valid for Lebesgue almost every $\al\in\R$.

From now on, $(q_n)$ will denote a subsequence of (instead of all) continuants of $\al$ such that
\begin{equation}\label{continued fractions 0}
\lim_{n\rightarrow\infty}q_{n}\|q_n\al\|=0
\end{equation}
and, whenever $\al$ is divisible, this chosen sequence $(q_n)$ will also satisfy that $2q_n$
divides $q_{n+1}$. We will also make constant use of the following conditions:
\begin{enumerate}
\item[(CF1)] For any $n\ge 1$,
$$2\sum_{j>n}\|q_j\al\|<\|q_n\al\|.$$
\item[(CF2)] For any $p\ge 1$,
$$\sum_{j\ge 1}j^{p+1}\cdot\|q_j\al\|<\infty.$$
\item[(CF3)] For any $p\ge 1$,
$$\sum_{j=1}^n j^{p+1}\cdot q_j<q_{n+1}\ \text{ for }n>n(p).$$
\item[(CF4)] For any $n\ge 1$,
$$\left(2^n\sum_{j=1}^{n-1}q_j\right)\cdot q_n\|q_n\al\|<1.$$
\item[(CF5)] For any $n\ge 1$, $\{\al,2\al,\ldots,q_{n+1}\al\}$ is $\left(\frac{1}{2q_n}\right)^2$-dense in $\T$.
\end{enumerate}

Condition (CF2) is always satisfied. Indeed,
$$\sum_{j\ge 1}j^{p+1}\cdot\|q_j\al\|<\sum_{j\ge 1}\dfrac{j^{p+1}}{q_j}$$
is bounded for every $p\ge 1$, because the exponential behavior of $q_j$ controls the
polynomial behavior of $j^{p+1}$. (CF1), (CF3), (CF4) and (CF5) are assured by passing, if
necessary, to a subsequence of $(q_n)$.

\subsection{Birkhoff sums}

Let $\al\in\R$, $\phi:\T\rightarrow\R$ a $L^1$-measurable function and $F$ defined
as in (\ref{intro - eq 1}). The dynamics of $F$ is intimately connected to the cocycle
$S(\al,\phi):\T\times\Z\rightarrow\R$ defined as the Birkhoff sums of $\phi$ with respect to
the rotation $R_\al$:
$$
S(\al,\phi)(x,n)=\left\{
\begin{array}{ll}
\displaystyle\sum_{k=0}^{n-1}\phi(x+k\al)&,\text{if }n\ge 1\\
&\\
0                                        &,\text{if }n=0   \\
&\\
-\displaystyle\sum_{k=1}^{-n}\phi(x-k\al)&,\text{if }n<0.\\
\end{array}\right.
$$
For simplicity, we denote $S(\al,\phi)(\,\cdot\,,n):\T\rightarrow\R$ by $S_n(\al,\phi)$.
From now on, we assume $\int_\T \phi d\lambda=0$. This in particular implies conservativity
of the associated cylinder flow. See \S 8.1 of \cite{Aa4}.

\subsection{Law of large numbers}\label{sub lln}

As observed in the introduction, there is no Birkhoff-type theorem for ergodic
and infinite measure-preserving systems. Nevertheless, one can hope for a law of
large numbers.

\begin{definition}
A {\it law of large numbers} for a conservative ergodic measure-preserv\-ing system
$(X,\mathcal A,\mu,F)$ is a function $L:\{0,1\}^\N\rightarrow[0,\infty]$ such that,
for any $A\in\mathcal A$, the equality
$$L(\ind_A(x),\ind_A(Fx),\ind_A(F^2x),\ldots)=\mu(A)$$
holds for $\mu$-almost every $x\in X$.
\end{definition}

One can see the function $L$ as a sort of blackbox: given the input of hittings of a generic
point $x\in X$ to a fixed set $A\in\mathcal A$, the output is the measure of $A$. There are
systems with no law of large numbers\footnote{Squashable transformations, for example, have no law
of large numbers. See \S 8.4 of \cite{Aa4}.}. On the other hand, there are some conditions that
guarantee its existence.

Given $A\in\mathcal A$, let $S_n(A):X\rightarrow\N$ be the Birkhoff sum
of the characteristic function $\ind_A$ with respect to $F$.

\begin{definition}\label{def rational ergodicity}
A conservative ergodic measure-preserving system $(X,\mathcal A,\mu,F)$ is called
{\it rationally ergodic along a subsequence of iterates} if there is a set $A\in\mathcal A$
with $0<\mu(A)<\infty$ satisfying the {\it Renyi inequality}
$$\int_A S_{n_k}(A)^2d\mu \ \lesssim\ \left(\int_A S_{n_k}(A)d\mu\right)^2$$
for some increasing sequence $(n_k)$ of positive integers.
\end{definition}

We note the above definition differs from the original one \cite{Aa2}, since the Renyi
inequality is asked to hold, instead of all positive integers, only for a subsequence
of them.

\begin{definition}
A conservative ergodic measure-preserving system $(X,\mathcal A,\mu,F)$ is called
{\it weakly homogeneous} if there is a sequence $(a_{n_k})$ of positive real numbers
such that, for all $\phi\in L^1(X,\mathcal A,\mu)$,
\begin{equation}\label{iet - eq 1}
\dfrac{1}{N}\sum_{k=1}^N\dfrac{1}{a_{n_k}}\sum_{j=0}^{n_k-1}\phi\left(F^jx\right)\ \longrightarrow\ \int_X\phi d\mu
\end{equation}
for $\mu$-almost every $x\in X$.
\end{definition}

$(a_{n_k})$ is called a {\it return sequence} of $F$ and it is unique up to asymptotic equality.
Aaronson proved that rational ergodicity along a subsequence of iterates implies weak homogeneity with
\begin{equation}\label{renyi - eq 1}
a_{n_k}=\dfrac{1}{\mu(A)^2}\int_A S_{n_k}(A)d\mu =\dfrac{1}{\mu(A)^2}\sum_{j=0}^{n_k-1}\mu\left(A\cap F^{-j}A\right).
\end{equation}
See \S 3.3 of \cite{Aa4}. Observe that weak homogeneity defines a law of large numbers
$L:\{0,1\}^\N\rightarrow[0,\infty]$ by
$$L(x_0,x_1,\ldots)=\left\{
\begin{array}{rl}
\displaystyle\lim_{N\rightarrow\infty}\dfrac{1}{N}\sum_{k=1}^N\dfrac{1}{a_{n_k}}\sum_{j=0}^{n_k-1}x_j\ ,&\text{ if the limit exists,}\\
&\\
0\ ,&\text{ otherwise.}
\end{array}
\right.$$

The goal of this work is to construct examples of cylinder flows given by skew product extensions
of irrational rotations on the circle that are ergodic and rationally ergodic along a subsequence
of iterates and, therefore, have law of large numbers.

\section{Construction of roof function $\phi$}\label{section - construction of function}

Let $T:\T\rightarrow\Z$ be the {\it Haar function}, defined as
$$T(x)=\left\{
\begin{array}{rl}
1\ ,&\text{ if }x\in\left[0,\dfrac{1}{2}\right)\\
&\\
-1\ ,&\text{ if }x\in\left[\dfrac{1}{2},1\right)\,.
\end{array}
\right.$$

\begin{center}
\psset{unit=1.5cm} \begin{pspicture}(-.6,-2)(2.8,1.6)

\psline[linewidth=.5pt]{->}(-.6,0)(2.8,0)\psline[linewidth=.5pt]{->}(0,-1.4)(0,1.4)
\psline[linewidth=1pt,linecolor=red]{*-o}(0,1)(1,1)
\psline[linewidth=1pt,linecolor=blue]{*-o}(1,-1)(2,-1)
\psline[linestyle=dotted](1,1)(1,-1)
\psline[linestyle=dotted](0,-1)(1,-1)
\psline[linestyle=dotted](2,0)(2,-1)
\uput[180](0,1){1}\uput[180](0,-1){-1}\uput[90](2,0){1}\uput[45](1,0){$\frac{1}{2}$}
\uput[180](2.6,-1.7){Figure 1: the graph of $T$.}
\end{pspicture}\end{center}

Let $\al\in\R$ and $(q_n)$ its associated subsequence of continuants, that is, satisfying
(\ref{continued fractions 0}) and (CF1) to (CF4). For each $j\ge 1$, let $T_j:\T\rightarrow\Z$ be the dilation of
$T$ by $q_j$, that is, $T_j(x)=T(q_jx)$, where $q_jx$ (and any expression appearing as
argument of $T$) is taken modulo $1$. The function we will consider is
$$\phi(x)=\dfrac{1}{2}\sum_{j\ge 1}\Big[T_j(x+\al)-T_j(x)\Big].$$
First of all, it is not clear that this defines a $L^1$-measurable function.
The proof of this fact depends on a couple of auxiliary lemmas.

\begin{lemma}\label{lemma - auxiliary 1}
Let $q$ be a positive integer and $\beta,\gamma\in\T$. Then the set
$$\{x\in\T\,;\,T(qx+\beta)\not=T(qx+\gamma)\}$$
has Lebesgue measure equal to $2\|\beta-\gamma\|$.
\end{lemma}

\begin{center}
\psset{unit=1.5cm} \begin{pspicture}(-.6,-2)(5,1.6)

\psline[linewidth=.5pt]{->}(-.6,0)(5,0)\psline[linewidth=.5pt]{->}(0,-1.4)(0,1.4)
\psline[linewidth=1pt,linecolor=red]{*-o}(0,1)(.5,1)
\psline[linewidth=1pt,linecolor=red]{*-o}(1,1)(1.5,1)
\psline[linewidth=1pt,linecolor=red]{*-o}(3,1)(3.5,1)
\psline[linewidth=1pt,linecolor=blue]{*-o}(.5,-1)(1,-1)
\psline[linewidth=1pt,linecolor=blue]{*-o}(1.5,-1)(2,-1)
\psline[linewidth=1pt,linecolor=blue]{*-o}(3.5,-1)(4,-1)
\psline[linestyle=dotted](.5,1)(.5,-1)
\psline[linestyle=dotted](0,-1)(.5,-1)
\psline[linestyle=dotted](1,1)(1,-1)
\psline[linestyle=dotted](1.5,1)(1.5,-1)
\psline[linestyle=dotted](3.5,1)(3.5,-1)
\psline[linestyle=dotted](2,0)(2,-1)
\psline[linestyle=dotted](3,0)(3,1)
\psline[linestyle=dotted](4,0)(4,-1)
\uput[180](0,1){1}\uput[180](0,-1){-1}
\uput[45](1,0){$\frac{2}{2q}$}\uput[45](.5,0){$\frac{1}{2q}$}
\uput[140](3,0){$\frac{2q-2}{2q}$}\uput[45](3.5,0){$\frac{2q-1}{2q}$}
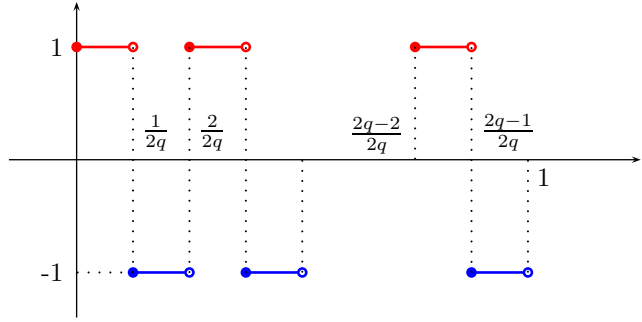
\uput[-45](4,0){1}\uput[180](4.3,-1.7){Figure 2: the graph of $x\mapsto T(qx)$.}
\end{pspicture}\end{center}

\begin{proof}
Just observe that, because the Lebesgue measure is preserved under the map
$x\mapsto qx$, $\{x\in\T\,;\,T(qx+\beta)\not=T(qx+\gamma)\}$
has Lebesgue measure equal to the Lebesgue measure of the set
$\{x\in\T\,;\,T(x+\beta)\not=T(x+\gamma)\}$, which is equal to $2\|\beta-\gamma\|$.
\end{proof}

\begin{lemma}\label{lemma - auxiliary 2}
Let $(q_n)$ be a sequence of positive integers and $(\beta_n)$, $(\gamma_n)$
sequences in $\T$. If $\psi:\T\rightarrow\Z$ is defined by
$$\psi(x)=\dfrac{1}{2}\sum_{j\ge 1}\Big[T(q_jx+\beta_j)-T(q_jx+\gamma_j)\Big],$$
then
\begin{equation}\label{convergence - eq 1}
\left\|\psi\right\|_p^p\le 2\sum_{j\ge 1}j^{p+1}\cdot\|\beta_j-\gamma_j\|.
\end{equation}
\end{lemma}

\begin{proof}
Assume the right hand side of (\ref{convergence - eq 1}) is finite. In particular,
$\sum\|\be_j-\gamma_j\|$ is convergent. For each $n\ge 1$, let
$$\Lambda_n=\{x\in\T\,;\,T(q_jx+\beta_j)=T(q_jx+\gamma_j),\ \forall\,j>n\}.$$
In $\Lambda_n$, we have
$$\psi(x)=\dfrac{1}{2}\sum_{j=1}^n\Big[T(q_jx+\beta_j)-T(q_jx+\delta_j)\Big].$$
The complement of $\Lambda_n$ is defined by the property that $T(q_jx+\beta_j)\not=T(q_jx+\gamma_j)$
for some $j>n$. By Lemma \ref{lemma - auxiliary 1}, its Lebesgue measure is at most
$2\sum_{j>n}\|\be_j-\gamma_j\|$. Then the sequence of functions $(\psi_n)$ given by
$\psi_n=\psi\cdot\ind_{\Lambda_n}$ converges pointwise to $\psi$. By Fatou's Lemma, the
result will follow if we manage to prove (\ref{convergence - eq 1}) for each $\psi_n$.

Fixed $n\ge 1$, we have
$$|\psi_n(x)|\le\sum_{j=1}^n\left|\dfrac{T(q_jx+\beta_j)-T(q_jx+\gamma_j)}{2}\right|\,,\ \forall\, x\in\T.$$
Define, for each $m\in\{1,\ldots,n\}$, the set
$$A_m=\left\{x\in\T\,;\,\sum_{j=1}^n\left|\dfrac{T(q_jx+\beta_j)-T(q_jx+\gamma_j)}{2}\right|=m\right\}.$$
If we further define, for each $j\in\{1,\ldots,n\}$, the set
$$A_m^j=\left\{x\in A_m\,;\,j\text{ is the largest index such that }T(q_jx+\beta_j)\not=T(q_jx+\gamma_j)\right\},$$
then
$$A_m=\bigsqcup_{j=m}^n A_m^j.$$
Each $A_m^j$ is contained in the set $\{x\in\T\,;\,T(q_jx+\beta_j)\not=T(q_jx+\gamma_j)\}$
and so, by Lemma \ref{lemma - auxiliary 1}, its Lebesgue measure is at most $2\|\beta_j-\gamma_j\|$.
Summing up this estimate in $j$ and $m$, we obtain that
\begin{eqnarray*}
\left\|\psi_n\right\|_p^p&=  &\int_{\T}|\psi_n|^pd\la\\
                &\le&\sum_{m=1}^n m^p\cdot\la(A_m)\\
                &\le&2\sum_{m=1}^n m^p\sum_{j=m}^n\|\beta_j-\gamma_j\|\\
                &\le&2\sum_{m=1}^n\sum_{j=m}^n j^p\cdot\|\beta_j-\gamma_j\|\\
                &\le&2\sum_{j\ge 1}j^{p+1}\cdot\|\beta_j-\gamma_j\|\,,
\end{eqnarray*}
thus establishing (\ref{convergence - eq 1}) for $\psi_n$.
\end{proof}

Lemma \ref{lemma - auxiliary 2} will be used repeatedly in the next subsections, the first
time being to prove that $\phi_n$, as defined in (\ref{intro - eq 3}), converges to $\phi$.

\subsection{$(\phi_n)$ converges to $\phi$ in $L^p(\T)$}\label{sub phi is L^p}

By Lemma \ref{lemma - auxiliary 2},
\begin{eqnarray*}
\left\|\phi-\phi_n\right\|_p^p&=  &\left\|\dfrac{1}{2}\sum_{j>n}\Big[T(q_jx+q_j\al)-T(q_jx)\Big]\right\|_p^p\\
                              &\le&2\sum_{j>n}j^{p+1}\cdot\|q_j\al\|
\end{eqnarray*}
which, by condition (CF2), goes to zero as $n$ goes to infinity.

\subsection{$(\tilde{\phi}_n)$ converges to $\phi$ in $L^p(\T)$}\label{sub phi tilde}

In order to make the calculations of Section \ref{section - number of returns}, in which
estimates on the return map of $F$ will be given, we need to approximate $\phi$ by something
easier to manage with. We will approximate $\phi$ not by $\phi_n$, but by its ``rational''
truncated versions $\tilde\phi_n$, defined as
\begin{equation}\label{convergence - eq 2}
\tilde\phi_n(x)=\dfrac{1}{2}\sum_{j=1}^n\Big[T_j(x+\alpha_{n+1})-T_j(x)\Big].
\end{equation}

Let us prove that the functions $\tilde\phi_n$ converge to $\phi$ in $L^p(\T)$ for any $p\ge 1$.
This follows by another application of Lemma \ref{lemma - auxiliary 2}. Indeed, as
$$\phi_n(x)-\tilde\phi_n(x)=\dfrac{1}{2}\sum_{j=1}^n\Big[T(q_jx+q_j\al)-T(q_jx+q_j\al_{n+1})\Big],$$
we have
\begin{eqnarray*}
\left\|\tilde\phi_n-\phi_n\right\|_p^p&\le&2\sum_{j=1}^n j^{p+1}\cdot\|q_j\al-q_j\al_{n+1}\|\\
                             &\le&2\sum_{j=1}^n j^{p+1}\cdot q_j\cdot|\al-\al_{n+1}|\\
                             &=  &\dfrac{2\|q_{n+1}\al\|}{q_{n+1}}\sum_{j=1}^n j^{p+1}\cdot q_j\\
                             &\le&2\|q_{n+1}\al\|\,,
\end{eqnarray*}
where in the last inequality we used (CF3).

\section{Ergodicity}\label{section - ergodicity}

\subsection{Branches and plateaux}\label{sub branches and plateaux}

We call a {\it branch} of $T_j$ any of the branches $\left[\frac{i}{q_j},\frac{i+1}{q_j}\right)$,
$i=0,1,\ldots,q_j-1$, of the expanding map $x\mapsto q_jx$. Each branch of $T_j$ decomposes
itself in two subintervals $\left[\frac{2i}{2q_j},\frac{2i+1}{2q_j}\right)$ and
$\left[\frac{2i+1}{2q_j},\frac{2i+2}{2q_j}\right)$, each of them called a {\it plateau} of $T_j$,
in which $T_j$ is constant (see figure 2). The first will be called a {\it positive plateau}
and the second a {\it negative plateau}.

Let $I_j(x)$ denote the plateau of $T_j$ containing $x$ and
$$m_n(x):=T_1(x)+\cdots+T_n(x)\ , \ \ n\ge 1.$$
If $(q_n)$ satisfies the divisibility condition, then clearly $I_1(x)\supset I_2(x)\supset\cdots$ and so
we have the implication
\begin{equation}\label{ergo - eq 1}
y\in I_n(x)\ \Longrightarrow\ m_n(x)=m_n(y).
\end{equation}
This is also true if, instead of the divisibility condition, $(q_n)$ satisfies Lemma \ref{lemma random walk}.
More specifically, using the notation of Appendix \ref{appendix random walk},
$$I_{n_0}(x)\supset I_{n_0+1}(x)\supset\cdots\ \ \text{ whenever }x\in \Omega_{n_0}^\infty.$$
For such a fixed $x$, there is a positive integer $n_1=n_1(x)$ such that
\begin{eqnarray*}
I_n(x)&\subset &I_1(x),\ldots,I_{n_0}(x)\\
\Longrightarrow\hspace{.7cm}\bigcap_{j=1}^n I_j(x)&=&I_n(x)
\end{eqnarray*}
for every $n\ge n_1$ and so (\ref{ergo - eq 1}) remains valid. We will use this condition below.

\subsection{Ergodicity}

We will prove ergodicity in two steps.\\

\noindent {\bf Step 1.} For any $A\subset\T\times\{0\}$ of positive measure, the union
$\bigcup_{n\ge 1}F^nA$ contains $\T\times\{0\}$ modulo zero.\\

\noindent {\bf Step 2.} $F(\T\times\{0\})\cap(\T\times\{1\})$ and $F(\T\times\{0\})\cap(\T\times\{-1\})$
have positive measure.\\

Once this is done, it is clear that $F$ will be ergodic. Actually, let $A\subset\T\times\Z$ be
$F$-invariant with positive measure. We can assume that $A$ has positive measure
when restricted to the fiber $\T\times\{0\}$. By Step 1, $A$ has full measure in $\T\times\{0\}$.
By Step 2, $A$ has also positive measure in both fibers $\T\times\{1\}$ and $\T\times\{-1\}$.
Applying repeatedly Steps 1 and 2, we conclude that $A$ has full measure in $\T\times\Z$.
Step 1 will follow from the next

\begin{lemma}\label{lemma - ergodicity 1}
Let $A_1,A_2\subset\T\times\{0\}$ have positive $\mu$-measure. Then there is $n\ge 1$ such that
the intersection $F^nA_1\cap A_2$ has positive $\mu$-measure.
\end{lemma}

To prove Lemma \ref{lemma - ergodicity 1}, we will localize $A_1$ and $A_2$ to subsets in
which $\phi$\ and $\phi_n$ coincide, and actually their Birkhoff sums up to the order $q_{n+1}$.
Letting $\mathcal D=\{0,1/2\}$, this set is defined as
$$\Lambda_n=\left\{x\in\T\,;\,d(q_jx,\mathcal D)>q_j\|q_j\al\|\text{ for }j>n\right\}.$$
Note that
$$d(q_j(x+k\al),q_jx)=\|kq_j\al\|=k\|q_j\al\|\le q_j\|q_j\al\|$$
whenever $j>n$ and $k=1,\ldots,q_{n+1}$. This implies that
$$F^k(x,0)=(x+k\al,S_k(\al,\phi_n)(x))\ \ ,\ x\in\Lambda_n,\ k=1,\ldots,q_{n+1}.$$
Observe that the $\Lambda_n$'s form an ascending chain of subsets of $\T$ and that
$\T\backslash\Lambda_n$ has Lebesgue measure at most $\sum_{j>n}q_j\|q_j\al\|$.
We can suppose, after passing to a subsequence\footnote{Here is where Theorem \ref{thm 1}
requires that $\liminf_{q\rightarrow\infty} q\|q\al\|=0$.}, that this sum is smaller than $2^{-n}$.

\begin{proof}[Proof of Lemma \ref{lemma - ergodicity 1}]
Define the set
$$\Sigma_n=\left\{x\in\T\,;\,d(x,\partial I_j(x))>\left(\dfrac{1}{2q_j}\right)^2\text{ for }j>n\right\}.$$
The sequence $(\Sigma_n)$ also forms an ascending chain of subsets of $\T$ and\footnote{For each plateau
of $T_j$, we remove two intervals of length $\left(\frac{1}{2q_j}\right)^2$. As $T_j$ has $2q_j$ plateaux,
the estimate is correct.}
$$\la(\T\backslash\Sigma_n)\le \sum_{j>n}\dfrac{1}{q_j}\,\cdot$$
This together with the fact that $\la(\Lambda_n),\la(\Omega_n^\infty)\rightarrow 1$ as $n\rightarrow\infty$
allows us to take $n_0\ge 1$ large enough and assume that
\begin{enumerate}
\item[(i)] $A_1\subset\Lambda_{n_0}$,
\item[(ii)] $A_1,A_2\subset\Sigma_{n_0}$ and
\item[(iii)] $A_1,A_2\subset \Omega_{n_0}^\infty$.
\end{enumerate}
By the Lebesgue differentiation theorem, let $x_1,x_2$ be points of density for $A_1,A_2$, respectively.
Now choose $n_1\ge 1$ large enough (see Subsection \ref{sub branches and plateaux}) such that
\begin{enumerate}
\item[(iv)] $\bigcap_{j=1}^n I_j(x_i)=I_n(x_i)$ for every $n\ge n_1$ and $i=1,2$.
\end{enumerate}
Finally, let $n\ge n_0,n_1$ such that
\begin{enumerate}
\item[(v)] $m_n(x_1)=m_n(x_2)$ and
\item[(vi)] $\la\left(A_i\cap\left(x_i-\left(\frac{1}{2q_n}\right)^2,x_i+
\left(\frac{1}{2q_n}\right)^2\right)\right)>\frac{3}{4}\cdot 2\left(\frac{1}{2q_n}\right)^2$ for $i=1,2$.
\end{enumerate}
The existence of such $n$ is assured by Lemma \ref{lemma random walk} and the fact that $x_i$ is
a point of density for $A_i$. For simplicity, let
$$\tilde A_i=A_i\cap\left(x_i-\left(\frac{1}{2q_n}\right)^2,x_i+\left(\frac{1}{2q_n}\right)^2\right),\ \ i=1,2.$$
By (ii), $\tilde A_i\subset I_n(x_i)$. Now use (CF5) to choose $k\in\{1,2,\ldots,q_{n+1}\}$ such that
\begin{equation}\label{ergo - eq 2}
d(x_1+k\al,x_2)<\left(\dfrac{1}{2q_n}\right)^2\,\cdot
\end{equation}

\noindent The proof of the lemma will follow from the next two claims.\\

\noindent {\bf Claim 1.} The set $(\tilde A_1+k\al)\cap\tilde A_2\subset\T$ has positive Lebesgue measure.\\

Indeed, (\ref{ergo - eq 2}) implies that the union $(\tilde A_1+k\al)\cup\tilde A_2$ is contained in an
interval of length $3\cdot\left(\frac{1}{2q_n}\right)^2$ and so, by (vi),
\begin{eqnarray*}
\la((\tilde A_1+k\al)\cap\tilde A_2)&=&\la(\tilde A_1+k\al)+\la(\tilde A_2)-\la((\tilde A_1+k\al)\cup\tilde A_2)\\
                &>&\frac{3}{2}\cdot\left(\dfrac{1}{2q_n}\right)^2+\frac{3}{2}\cdot\left(\dfrac{1}{2q_n}\right)^2
                   -3\cdot\left(\dfrac{1}{2q_n}\right)^2\\
                &=&0\,.
\end{eqnarray*}

\noindent {\bf Claim 2.} The set $F^k(\tilde A_1\times\{0\})\cap(\tilde A_2\times\{0\})\subset\T\times\Z$
has positive $\mu$-measure.\\

It is enough to prove that $S_k(\al,\phi)(x)=0$ for every $x$ satisfying Claim 1.
By (i), $x\in\Lambda_{n_0}\subset\Lambda_n$ and so
$$S_k(\al,\phi)(x)=S_k(\al,\phi_n)(x)=m_n(x+k\al)-m_n(x).$$
Observe that
\begin{enumerate}[$\bullet$]
\item $x\in\tilde A_1\subset I_n(x_1)$ and so (iv) guarantees that $m_n(x)=m_n(x_1)$.
\item $x+k\al\in\tilde A_2\subset I_n(x_2)$. Using (iv) again, $m_n(x+k\al)=m_n(x_2)$.
\end{enumerate}
By assumption (v) it follows that $S_k(\al,\phi)(x)=0$ for every $x$ satisfying Claim 1.
This concludes the proof of Claim 2 and also from the lemma.
\end{proof}

We thus obtained Step 1. Step 2 follows from Lemma \ref{lemma - auxiliary 1}. Indeed,
for $s\in\{-1,1\}$, the set of points $x\in\T$ such that
\begin{enumerate}[$\bullet$]
\item $T_1(x+\al)=T_1(x)+2s$ and
\item $T_j(x+\al)=T_j(x)$ for $j>1$
\end{enumerate}
has Lebesgue measure at least $\|q_1\al\|-2\sum_{j>1}\|q_j\al\|$, which is positive by (CF1).
This concludes the proof of ergodicity.

\begin{remark}
The argument of this section indeed shows that $F$ is {\it regular} in the sense of \cite{Sc77}.
\end{remark}

\section{Counting the number of returns}\label{section - number of returns}

Throughout this and the next section, we assume $\al$ is divisible and its subsequence $(q_n)$ of
continuants satisfies
$$2q_n\text{ divides }q_{n+1}\ \ \text{ and }\ \ \lim_{n\rightarrow\infty}q_n\|q_n\al\|=0.$$

Let $A=\T\times\{0\}$. The purpose of this section is to count the number of returns of an arbitrary
point $(x,0)\in A$ to $A$ via the map $F$. More specifically, identifying $A$ with $\T$, we want to
investigate the function $S^F_{q_{n+1}}:\T\rightarrow\N$ defined as
$$S^F_{q_{n+1}}(x)=\sum_{k=1}^{q_{n+1}}(\ind_A\circ F^k)(x,0).$$
In the next section we will apply the estimates obtained here to establish Theorem \ref{thm 1}.

As remarked before, we will not directly calculate $S^F_{q_{n+1}}$. Instead, we consider the rational
truncated versions of $F$ defined by the skew product
$$
\begin{array}{rcrcl}
\tilde F_n&:&\T\times\Z&\longrightarrow &\T\times\Z\\
          & & (x,y)    &\longmapsto     &(x+\al_{n+1},y+\tilde\phi_n(x)),
\end{array}
$$
where $\tilde\phi_n$ is given by (\ref{convergence - eq 2}), and calculate the value of
$S^{\tilde F_n}_{q_{n+1}}:\T\rightarrow\N$ given by
$$S^{\tilde F_n}_{q_{n+1}}(x)=\sum_{k=1}^{q_{n+1}}(\ind_A\circ{\tilde F_n}^k)(x,0).$$
By approximation, $S^F_{q_{n+1}}$ and $S^{\tilde F_n}_{q_{n+1}}$ coincide
for a large subset of $\T$ and then we will have the value of the former function in this large set.

This section is organized as follows. In Subsection \ref{sub counting lemma}, we calculate the
distribution of $S^{\tilde F_n}_{q_{n+1}}$. After that, Subsection \ref{sub returns for F}
establishes the distribution of $S^F_{q_{n+1}}$.

\subsection{The function $S^{\tilde F_n}_{q_{n+1}}$}\label{sub counting lemma}

Observe that
$${\tilde F_n}^k(x,0)=(x+k\al_{n+1},S_k(\al_{n+1},\tilde \phi_n)(x))$$
so that ${\tilde F_n}^k(x,0)$ belongs to $A$ if and only if
$$S_k(\al_{n+1},\tilde\phi_n)(x)=0\ \iff\ m_n(x+k\al_{n+1})=m_n(x).$$
Then
$$S^{\tilde F_n}_{q_{n+1}}(x)=\#\{1\le k\le q_{n+1}\,;\,m_n(x+k\al_{n+1})=m_n(x)\}.$$

The idea to calculate the above cardinality is: for each sequence
${\bf s}=(s_1,\ldots,s_n)\in\{-1,1\}^n$, consider the set
$$B_{\bf s}=\{1\le k\le q_{n+1}\,;\, T_j(x+k\al_{n+1})=s_j\text{ for }j=1,\ldots,n\}.$$
If we manage to prove that each $B_{\bf s}$ has the same cardinality (independent of ${\bf s}$),
it must be equal to $q_{n+1}/2^n$. Then
\begin{eqnarray*}
S^{\tilde F_n}_{q_{n+1}}(x)&=&\sum_{{\bf s}\in\{-1,1\}^n\atop{s_1+\cdots+s_n=m_n(x)}}\#B_{\bf s}\\
                           &=&\dfrac{q_{n+1}}{2^n}\cdot\#\{{\bf s}\in\{-1,1\}^n\,;\,s_1+\cdots+s_n=m_n(x)\}
\end{eqnarray*}
and so
\begin{equation}\label{counting lemma - eq 1}
S^{\tilde F_n}_{q_{n+1}}(x)=\dfrac{q_{n+1}}{2^n}{n\choose{\frac{n+m_n(x)}{2}}}\cdot
\end{equation}
This is indeed the case. Roughly speaking, we prove that each $B_{\bf s}$ has the same cardinality by
interpreting $m_n(x)$ as a random walk. More specifically, we consider the intermediate sets
$$B_{(s_1,\ldots,s_i)}=\{1\le k\le q_{n+1}\,;\, T_j(x+k\al_{n+1})=s_j\text{ for }j=1,\ldots,i \}$$
and associate to them a binary tree as follows:
\begin{enumerate}[$\bullet$]
\item The root of the tree is $B=\{1,2,\ldots,q_{n+1}\}$.
\item $B_{(s_1,\ldots,s_i)}$ has exactly two descendants: $B_{(s_1,\ldots,s_i,1)}$ and $B_{(s_1,\ldots,s_i,-1)}$.
\end{enumerate}
Observe that
$$B_{(s_1,\ldots,s_i)}=B_{(s_1,\ldots,s_i,1)}\sqcup B_{(s_1,\ldots,s_i,-1)}$$
so that, at each level $i$, the union of the $B_{(s_1,\ldots,s_i)}$'s is equal to $B$. We will prove that,
in each subdivision of $B_{(s_1,\ldots,s_i)}$, half of the elements belong to $B_{(s_1,\ldots,s_i,1)}$
and the other half to $B_{(s_1,\ldots,s_i,-1)}$. Once this is done, (\ref{counting lemma - eq 1}) will
be established.

\begin{center}
\psset{unit=.65cm} \begin{pspicture}(-9,-1.7)(9,7)

\psline{*->}(0,6)(-5,4)\psline{*->}(0,6)(5,4)
\psline{*->}(-5,4)(-7,2)\psline{*->}(-5,4)(-3,2)\psline{*->}(5,4)(3,2)\psline{*->}(5,4)(7,2)
\psline{*->}(-7,2)(-8,0)\psline{*->}(-7,2)(-6,0)\psline{*->}(-3,2)(-4,0)\psline{*->}(-3,2)(-2,0)
\psline{*->}(3,2)(2,0)\psline{*->}(3,2)(4,0)\psline{*->}(7,2)(6,0)\psline{*->}(7,2)(8,0)
\psline{*-*}(-8,0)(-8,0)\psline{*-*}(-6,0)(-6,0)\psline{*-*}(-4,0)(-4,0)\psline{*-*}(-2,0)(-2,0)
\psline{*-*}(8,0)(8,0)\psline{*-*}(6,0)(6,0)\psline{*-*}(4,0)(4,0)\psline{*-*}(2,0)(2,0)

\uput[90](0,6){$B$}\uput[180](-5,4){$B_{1}$}\uput[0](5,4){$B_{-1}$}
\uput[180](-7,2){$B_{(1,1)}$}\uput[0](-3,2){$B_{(1,-1)}$}
\uput[180](3,2){$B_{(-1,1)}$}\uput[0](7,2){$B_{(-1,-1)}$}
\uput[-90](-7,.5){$\vdots$}\uput[-90](-3,.5){$\vdots$}
\uput[-90](7,.5){$\vdots$}\uput[-90](3,.5){$\vdots$}
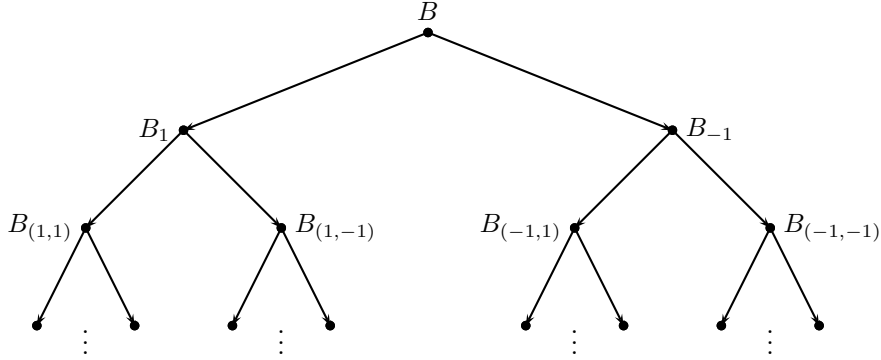
\uput[-90](0,-.8){Figure 3: the binary tree.}

\end{pspicture}\end{center}

Fix $x\in\T$. The idea is to see $k$ as a variable $z\in\R$ and to prove that the evaluations of the functions
$z\mapsto T_j(\al_{n+1}z+x)$, $j=1,2,\ldots,n$, along the integers $1,2,\ldots,q_{n+1}$ satisfy the
required binary property. Each of these functions is periodic, with period equal to
$$\pi_j=\dfrac{1}{q_j\al_{n+1}}=\dfrac{q_{n+1}/q_j}{p_{n+1}}=:\dfrac{u_j}{v}\,\cdot$$
Better than this, consider the functions given by the composition with the dilation $z\mapsto z/v$,
defined as
\begin{equation}\label{counting lemma - eq 3}
\begin{array}{rcccl}
\psi_j&:&\R&\longrightarrow &\R\\
      & &z &\longmapsto     &T\left(\dfrac{z}{u_j}+q_jx\right)\\
\end{array},\ \ j=1,2,\ldots,n\,,\\
\end{equation}
whose period is equal to $u_j\in\Z$. We thus want to investigate $\psi_1,\ldots,\psi_n$ along
the integers $v,2v,\ldots,u_1v$. Observe that
\begin{enumerate}[$\bullet$]
\item $\{v,2v,\ldots,u_1v\}$ is a complete residue system modulo $u_1$,
\item $u_n$ is even and $u_j$ is a multiple of $2u_{j+1}$ for $j=1,\ldots,n-1$, and
\item for a set $x\in\T$ of full Lebesgue measure, $\psi_1,\ldots,\psi_n$ are continuous in $\Z$ (i.e.
none of their discontinuities is an integer).
\end{enumerate}
These are the assumptions we make below.

\begin{proposition}\label{proposition - counting}
Let $\psi_j:\R\rightarrow\R$ be a periodic function with period $u_j\in\Z$, $j=1,\ldots,n$. Assume that
\begin{itemize}
\item[(a)] $u_n$ is even and $u_j$ is a multiple of $2u_{j+1}$ for $j=1,\ldots,n-1$, and
\item[(b)] there are $z_1,\ldots,z_n\in\R\backslash\Q$ such that
$$\psi_j|_{\left[z_j,z_j+\frac{u_j}{2}\right)}\equiv 1\ \text{ and }\
\psi_j|_{\left[z_j+\frac{u_j}{2},z_j+u_j\right)}\equiv -1$$
for $j=1,\ldots,n$.
\end{itemize}
Let $R$ be a complete residue system modulo $u_1$. Then, for any sequence
$(s_1,\ldots,s_n)\in\{-1,1\}^n$,
$$\#\{k\in R\,;\,\psi_j(k)=s_j\text{ for }j=1,\ldots,n\}=\dfrac{u_1}{2^n}\,\cdot$$
\end{proposition}

The proof is by induction on $n$. Let us give an idea of why this must be true. Assume that $x=0$
and that, instead of being interested in the behavior of $\psi_1,\ldots,\psi_n$ along integers,
we want to compute the Lebesgue measure of the set
\begin{equation}\label{counting lemma - eq 5}
\{z\in [0,u_1)\,;\,\psi_j(k)=s_j\text{ for }j=1,\ldots,n\}.
\end{equation}
For $n=1$, we have
\begin{eqnarray*}
\{z\in [0,u_1)\,;\,\psi_1(k)=1\}&=&\left[0,\frac{u_1}{2}\right)\\
\{z\in [0,u_1)\,;\,\psi_1(k)=-1\}&=&\left[\frac{u_1}{2},u_1\right).
\end{eqnarray*}
For $n=2$, observe that in both intervals $\left[0,\frac{u_1}{2}\right)$, $\left[\frac{u_1}{2},u_1\right)$ the
function $\psi_2$ alternately changes sign at each interval of length $u_2/2$ so that, for any $s_1,s_2\in\{-1,1\}$,
$\{z\in [0,u_1)\,;\,\psi_j(k)=s_j\text{ for }j=1,2\}$ is the union of $u_1/2u_2$ intervals of length $u_2/2$.
For arbitrary $n$, (\ref{counting lemma - eq 5}) is the union of $u_1/2^{n-1}u_n$ intervals of length $u_{n}/2$
each and so its Lebesgue measure is equal to $u_1/2^n$. Proposition \ref{proposition - counting} is nothing
but a discrete version of this. In order to prove it, we just have to make sure that none of the discontinuities
of $\psi_1,\ldots,\psi_n$ are integer. This is accomplished by condition (b).

The next auxiliary lemma constitutes the basis of induction.

\begin{lemma}\label{lemma - auxiliary 4}
Let $\psi:\R\rightarrow\R$ be a function with period $u\in\Z$ such that
\begin{itemize}
\item[(a)] $u$ is even and
\item[(b)] there is $z\in\R\backslash\Q$ such that
$$\psi|_{\left[z,z+\frac{u}{2}\right)}\equiv 1\ \text{ and }\
\psi|_{\left[z+\frac{u}{2},z+u\right)}\equiv -1.$$
\end{itemize}
Let $R$ be a complete residue system modulo $u$. Then
$$\#\{k\in R\,;\,\psi(k)=1\}=\#\{k\in R\,;\,\psi(k)=-1\}=\dfrac{u}{2}\,\cdot$$
\end{lemma}

\begin{proof}
Consider the sets
$$
\begin{array}{rcl}
\Psi_{+}&=&\left\{i\in\Z\,;\,i\in\left[z,z+\dfrac{u}{2}\right)\right\}\hspace{.6cm}\pmod u
\ \ \text{and}\\
&&\\
\Psi_{-}&=&\left\{i\in\Z\,;\,i\in\left[z+\dfrac{u}{2},z+u\right)\right\}\pmod u\,.
\end{array}
$$
It is clear that $\Psi_{+}\cup\Psi_{-}=\Z_u$ and that
$\#\Psi_{+}=\#\Psi_{-}=u/2$. Also, $\psi(k)=1$ if and only if
$k\equiv i\pmod u$ for some $i\in\Psi_{+}$. Because $R$ is a complete residue system module $u$,
the lemma is proved.
\end{proof}

\begin{proof}[Proof of Proposition \ref{proposition - counting}]
The basis of induction is Lemma \ref{lemma - auxiliary 4}. It remains to prove the inductive step.
We will do the case $n=2$, as the general inductive step follows the same lines of ideas,
except that more notation would have to be introduced.

Let $\psi_1,\psi_2:\R\rightarrow\R$ be two functions satisfying the conditions of the proposition.
For $j=1,2$, consider the equipartition of $\Z_{u_j}$ by the subsets
$$
\begin{array}{rcl}
\Psi_{+}^j&=&\left\{i\in\Z\,;\,i\in\left[z_j,z_j+\dfrac{u_j}{2}\right)\right\}\hspace{.8cm}\pmod {u_j}\ \ \text{ and}\\
&&\\
\Psi_{-}^j&=&\left\{i\in\Z\,;\,i\in\left[z_j+\dfrac{u_j}{2},z_j+u_j\right)\right\}\pmod {u_j}\,.
\end{array}
$$
For $s_1,s_2\in\{-1,1\}\cong\{-,+\}$,
$$
\left\{
\begin{array}{c}
\psi_1(k)=s_1\\
             \\
\psi_2(k)=s_2
\end{array}\right.
\ \iff\
\left\{
\begin{array}{c}
k\equiv i_1\pmod {u_1}\ \text{ for }i_1\in\Psi_{s_1}^1\\
                                                            \\
\,k\equiv i_2\pmod {u_2}\ \text{ for }i_2\in\Psi_{s_2}^2.
\end{array}\right.
$$
Because $u_2$ divides $u_1$, residue classes module $u_1$ define residue classes module $u_2$. This implies
that the above congruences are equivalent to
$$
\left\{
\begin{array}{c}
\,k\equiv i_1\pmod {u_1}\ \text{ for }i_1\in\Psi_{s_1}^1\\
                                                       \\
i_1\equiv i_2\pmod {u_2}\ \text{ for }i_2\in\Psi_{s_2}^2
\end{array}\right.
$$
and then we want to count the cardinality of the set
\begin{equation}\label{counting lemma - eq 2}
\left\{k\in R\,;\,
\begin{array}{c}
\,k\equiv i_1\pmod {u_1}\ \text{ for }i_1\in\Psi_{s_1}^1\\
i_1\equiv i_2\pmod {u_2}\ \text{ for }i_2\in\Psi_{s_2}^2
\end{array}\right\}\cdot
\end{equation}
Each residue class modulo $u_2$ is equal to the
union of $u_1/u_2$ residue classes modulo $u_1$. More specifically,
$$i_1\equiv i_2\pmod{u_2}\iff i_1\equiv i_2,i_2+u_2,\ldots,i_2+(u_1-u_2)\pmod{u_1}$$
so that (\ref{counting lemma - eq 2}) is equal to the union
$$\bigcup_{i_2\in\Psi_{s_2}^2}\{i_2,i_2+u_2,\ldots,i_2+(u_1-u_2)\}\cap\Psi_{s_1}^1.$$
Independent of $i_2$, half of the residue classes $i_2,i_2+u_2,\ldots,i_2+(u_1-u_2)$ modulo $u_1$ belong
to $\Psi_{+}^1$ and half to $\Psi_{-}^1$. Thus
\begin{eqnarray*}
\#\{k\in R\,;\,\psi_1(k)=s_1\text{ and }\psi_2(k)=s_2\}&=&\#\Psi_{s_2}^2\cdot\dfrac{u_1}{2u_2}\\
                                                              &=&\dfrac{u_2}{2}\cdot\dfrac{u_1}{2u_2}\\
                                                              &=&\dfrac{u_1}{4}\,,
\end{eqnarray*}
where in the second equality we used Lemma \ref{lemma - auxiliary 4}.
\end{proof}

In our context, Proposition \ref{proposition - counting} is translated to

\begin{lemma}\label{proposition - returns for rational}
For every $m\in\{-n,\ldots,n\}$ with the same parity of $n$,
$$S^{\tilde F_n}_{q_{n+1}}(x)=\dfrac{q_{n+1}}{2^n}{n\choose{\frac{n+m}{2}}}$$
for a set of $x\in\T$ of Lebesgue measure ${n\choose\frac{n+m}{2}}/2^n$.
\end{lemma}

\begin{proof}
Let $u_j=q_{n+1}/q_j$ for $j=1,\ldots,n$ and apply Proposition \ref{proposition - counting} to the
functions in (\ref{counting lemma - eq 3}). The random variable $x\in\T\mapsto m_n(x)$ has the same distribution
as the $n$-th step of a simple random walk in $\Z$, and so the equality $m_n(x)=m$ holds in a set of Lebesgue measure
${n\choose\frac{n+m}{2}}/2^n$, for every $m\in\{-n,\ldots,n\}$ with the same parity of $n$.
\end{proof}

\subsection{The function $S^{F}_{q_{n+1}}$}\label{sub returns for F}
It is a matter of fact that $\phi$ and ${\tilde\phi}_n$ coincide in a large set, and actually their
Birkhoff sums up to the order $q_{n+1}$. This set is defined by those
points simultaneously satisfying
\begin{enumerate}
\item[(i)] $T_j(x+k\al)=T_j(x+k\al_{n+1})$ for $j=1,\ldots,n$ and $k=1,\ldots,q_{n+1}$, and
\item[(ii)] $d(q_jx,\mathcal D)>q_j\|q_j\al\|$ for $j>n$.
\end{enumerate}
Call this set $\Lambda_n$. Note that
$$d(q_j(x+k\al),q_jx)=\|kq_j\al\|=k\|q_j\al\|\le q_j\|q_j\al\|$$
whenever $j>n$ and $k=1,\ldots,q_{n+1}$ and so (ii) implies $T_j(x+k\al)=T_j(x)$. This equality
guarantees that
\begin{eqnarray}
F^k(x,0)&=&(x+k\al,S_k(\al,{\tilde \phi}_n)(x))\ \hspace{.5cm}\ \text{for }x\in\Lambda_n,\ k=1,\ldots,q_{n+1}\nonumber \\
\Longrightarrow\hspace{.7cm}S_{q_{n+1}}^F(x)&=&
         S_{q_{n+1}}^{{\tilde F}_n}(x)\hspace{2.8cm}\text{for }x\in\Lambda_n. \label{counting lemma - eq 4}
\end{eqnarray}
By Lemma \ref{lemma - auxiliary 1}, the Lebesgue measure of points not satisfying (i) is at most
$$\sum_{1\le k\le q_{n+1}\atop{1\le j\le n}}\|kq_j(\al-\al_{n+1})\|<
|\al-\al_{n+1}|\cdot q_{n+1}^2\cdot\sum_{j=1}^n q_j<2^{-n-1}\ ,$$
where in the last inequality we used (CF4). The points not satisfying (ii) have Lebesgue measure at
most\footnote{Remember we are assuming $\sum_{j>n}q_j\|q_j\al\|<2^{-n}$.}
$\sum_{j>n}q_j\|q_j\al\|<2^{-n}$ and so
\begin{equation}\label{estimate lambda_n}
\la(\T\backslash\Lambda_n)<2^{-n+1}\,.
\end{equation}
The above estimate will be used in the next section.

\section{Rational ergodicity along $(q_n)$}\label{section - proof of rat ergodicity}

It remains to prove that $F$ satisfies the Renyi inequality along $(q_n)$. This will be obtained
via the estimates of Section \ref{section - number of returns}. More specifically, we first prove,
as a consequence of Lemma \ref{proposition - returns for rational}, that the rational truncated
version $\tilde F_n$ of $F$ satisfies the Renyi inequality in the time $q_{n+1}$, uniformly in $n$.
We then prove that $\|S_{q_{n+1}}^F\|_1\approx\|S_{q_{n+1}}^{\tilde F_n}\|_1$ and
$\|S_{q_{n+1}}^F\|_2\approx\|S_{q_{n+1}}^{\tilde F_n}\|_2$, which allows us to push the Renyi
inequality to $F$.

\subsection{Renyi inequality for $\tilde F_n$}

By Lemma \ref{proposition - returns for rational},
\begin{eqnarray*}
\left\|S^{\tilde F_n}_{q_{n+1}}\right\|_1&=&\displaystyle\int_{\T}S^{\tilde F_n}_{q_{n+1}}d\la\\
      &=&\sum_{-n\le m\le n\atop{m\equiv n(\text{mod }2)}}\left[\dfrac{q_{n+1}}{2^n}{n\choose\frac{n+m}{2}}\right]\cdot
         \left[\dfrac{1}{2^n}{n\choose\frac{n+m}{2}}\right]\\
      &=&\dfrac{q_{n+1}}{2^{2n}}\sum_{i=0}^n{n\choose i}^2\\
      &=&\dfrac{q_{n+1}}{2^{2n}}{2n\choose n}\\
      &\approx&\frac{q_{n+1}}{2^{2n}}\cdot\frac{2^{2n}}{\sqrt{\pi n}}\\
      &=&\dfrac{q_{n+1}}{\sqrt{\pi n}}\ ,
\end{eqnarray*}
where in the fifth passage we used Stirling's formula\footnote{Stirling's formula states that
$n!\approx\sqrt{2\pi n}\left(\frac{n}{e}\right)^n$.} to estimate the central binomial coefficient.
On the other hand,
\begin{eqnarray*}
\left\|S^{\tilde F_n}_{q_{n+1}}\right\|_2^2
            &=&\sum_{-n\le m\le n\atop{m\equiv n(\text{mod }2)}}\left[\dfrac{q_{n+1}}{2^n}{n\choose\frac{n+m}{2}}\right]^2
               \cdot\left[\dfrac{1}{2^n}{n\choose\frac{n+m}{2}}\right]\\
            &=&\dfrac{q_{n+1}^2}{2^{3n}} \sum_{i=0}^n {n\choose i}^3\\
            &\le &\dfrac{q_{n+1}^2}{2^{3n}}{n\choose\frac{n}{2}}\sum_{i=0}^n{n\choose i}^2\\
            &=&\dfrac{q_{n+1}^2}{2^{3n}}{n\choose\frac{n}{2}}{2n\choose n}\\
            &\approx &\sqrt{2}\cdot\dfrac{q_{n+1}^2}{\pi n}
\end{eqnarray*}
and therefore
\begin{equation}\label{renyi - eq 2}
\dfrac{\left\|S^{\tilde F_n}_{q_{n+1}}\right\|_2}{\left\|S^{\tilde F_n}_{q_{n+1}}\right\|_1}\lesssim
\dfrac{\sqrt[4]{2}\cdot\dfrac{q_{n+1}}{\sqrt{\pi n}}}{\dfrac{q_{n+1}}{\sqrt{\pi n}}}\lesssim 1\,.
\end{equation}

\subsection{Renyi inequality for $F$}

Using (\ref{estimate lambda_n}),
$$\left|\left\|S^F_{q_{n+1}}\right\|_1-\left\|S^{\tilde F_n}_{q_{n+1}}\right\|_1\right|
\ \le\ \int_{\T\backslash\Lambda_n}\left|S^F_{q_{n+1}}-S^{\tilde F_n}_{q_{n+1}}\right|d\la
\ < \ q_{n+1}\cdot 2^{-n+1}$$
and so
$$\left|\dfrac{\left\|S^F_{q_{n+1}}\right\|_1}{\left\|S^{\tilde F_n}_{q_{n+1}}\right\|_1}-1\right|
\ \lesssim\ \dfrac{q_{n+1}\cdot 2^{-n+1}}{\dfrac{q_{n+1}}{\sqrt{\pi n}}}\ \approx\ 0,$$
proving that $\left\|S^F_{q_{n+1}}\right\|_1\approx\left\|S^{\tilde F_n}_{q_{n+1}}\right\|_1$.
Analogously,
$$\left|\dfrac{\left\|S^F_{q_{n+1}}\right\|_2^2}{\left\|S^{\tilde F_n}_{q_{n+1}}\right\|_2^2}-1\right|
\ \lesssim\ \dfrac{q_{n+1}^2\cdot 2^{-n+1}}{\dfrac{q_{n+1}^2}{\pi n}}\ \approx\ 0$$
and so $\left\|S^F_{q_{n+1}}\right\|_2\approx\left\|S^{\tilde F_n}_{q_{n+1}}\right\|_2$.
These two estimates, together with (\ref{renyi - eq 2}), guarantee that
$$\left\|S^F_{q_{n+1}}\right\|_2\lesssim \left\|S^F_{q_{n+1}}\right\|_1,$$
thus establishing the Renyi inequality for $F$ along $(q_n)$. This concludes the proof of
Theorem \ref{thm 1}.\\

We calculate the return sequence $(a_{q_n})$ for $F$. According to (\ref{renyi - eq 1}), it is given by
$$a_{q_{n+1}}=\left\|S^F_{q_{n+1}}\right\|_1\approx\left\|S^{\tilde F_n}_{q_{n+1}}\right\|_1\approx
\dfrac{q_{n+1}}{\sqrt{\pi n}}$$
and so for a fixed $x\in\Lambda_n$ the normalized averages
\begin{equation}\label{equation norm average}
\dfrac{S_{q_{n+1}}^F(x)}{a_{q_{n+1}}}\approx
\dfrac{\dfrac{q_{n+1}}{2^n}\displaystyle{n\choose{\frac{n+m_n(x)}{2}}}}{\dfrac{q_{n+1}}{\sqrt{\pi n}}}=
\dfrac{\displaystyle{n\choose{\frac{n+m_n(x)}{2}}}}{\dfrac{2^n}{\sqrt{\pi n}}}\approx
\sqrt{2}\cdot\dfrac{\displaystyle{n\choose{\frac{n+m_n(x)}{2}}}}{\displaystyle{n\choose{\frac{n}{2}}}}
\end{equation}
do not depend on the choice of $\al$ neither on the sequence $(q_n)$.

\section{Final comments}

\noindent {\bf 1.} In order to obtain ergodic cylinder flows on $\T\times\R$, one can consider a similar
construction to ours with roof function as in (\ref{intro - eq 2}), where $\be\in\R$ is irrational. In this
case, the image of the map is contained in $\T\times\{m+n\be\,;\,m,n\in\Z\}$, which is dense in $\T\times\R$.\\

\noindent {\bf 2.} So far, all the examples of rationally ergodic cylinder flows use
non-continuous roof functions. Another natural program is to construct examples with
continuous (even $C^1$ and $C^\infty$) roof functions. It seems to us that the same approach
developed in the present paper might work if one can interpret the sequence $(m_n)$ as defined in
Subsection \ref{sub branches and plateaux} from a random perspective.\\

\section*{Acknowledgments}

The authors are thankful to IMPA for the excellent ambient during the preparation of this
manuscript and to Alejandro Kocsard, Fran\c cois Ledrappier, Carlos Gustavo Moreira and Omri
Sarig for valuable comments and suggestions. This research was possible due to the support of
CNPq-Brazil, Faperj-Brazil and J. Palis 2010 Balzan Prize for Mathematics.

\appendix

\section{Random walks}\label{appendix random walk}

Let $T:\T\rightarrow\Z$ as defined in Section \ref{section - construction of function}.
For each sequence of positive integers $(q_n)$, we associate the sequence $(T_n)$ of
functions defined on $\T$ by $T_n(x)=T(q_nx)$. This appendix is devoted to the analysis
of the partial sums
$$m_n(x)=T_1(x)+\cdots+T_n(x)\ , \ \ n\ge 1.$$
The sequence $(m_n)$ defines a random walk in $\Z$ and we are particularly interested in
the interaction between different walks. We say that $(m_n)$ has the {\it level-crossing property} if,
for Lebesgue almost every $x,y\in\T$, there exist infinitely many $n$ such that $m_n(x)=m_n(y)$.

If we assume that $2q_n$ divides $q_{n+1}$ then every plateau of $T_n$ contains exactly
the same number of positive and negative plateaux of $T_{n+1}$. If this holds for every $n$
then, for any $s_1,\ldots,s_n\in\{-1,1\}$,
$$\la(\{x\in\T\,;\,T_j(x)=s_j\text{ for }j=1,\ldots,n\})=2^{-n}$$
and so the $(T_n)$ are independent and identically distributed (i.i.d). In this case $(m_n)$
is a simple random walk in $\Z$, and thus the map that associates to each pair $(x,y)\in\T\times\T$
the process $(m_n(x)-m_n(y))$ is a random walk in $\Z$ with finite support and zero mean.
In particular, $(m_n(x)-m_n(y))$ is recurrent almost surely (see \S 4.2 of \cite{Dur}),
and so $(m_n)$ has the level-crossing property.

The same might not be true if $2q_n$ does not divide $q_{n+1}$. On the other hand, if $q_{n+1}$
is much larger than $q_n$, almost every plateau of $T_{n+1}$ is entirely contained inside a
plateau of $T_n$ and so $(T_n)$ exhibits some sort of asymptotic independence. This is the
content of the next result, which is used in Section \ref{section - ergodicity} to prove ergodicity
when one does not have the divisibility condition.  The idea is to remove plateaux of $T_{n+1}$
not entirely contained inside plateaux of $T_n$ in such a way that independence holds in their
complement.

\begin{lemma}\label{lemma random walk}
Let $(q_n)$ be a sequence of positive integers and let $(T_n)$, $(m_n)$ be as above. If
$$\sum_{n\ge 1}\dfrac{q_n}{q_{n+1}}<\infty$$
then $(m_n)$ has the level-crossing property.
\end{lemma}

\begin{proof}
We will construct a descending chain of Borel sets $(\Omega_n)$ of $\T$ such that, restricted to $\Omega_n$,
the first $n$ functions $T_1,\ldots,T_n$ are i.i.d. A simple argument of induction will imply that
the $(T_n)$ are i.i.d in the intersection $\Omega^\infty=\bigcap_{n\ge 1}\Omega_n$.

The construction is by induction. Let $\F_n$ be the family of plateaux of $T_n$ and
$\F_n=\F_n^+\bigsqcup\F_n^-$ its decomposition in positive and negative plateaux, respectively. Assume that
$\Omega_1=\T,\ldots,\Omega_n$ have been constructed satisfying the following conditions.
\begin{enumerate}[(i)]
\item For $1\le j\le n$, there is a set $\G_j\subset\F_j$ such that
$\Omega_j=\bigcup_{J\in\mathcal G_j}J$.
\item For $1\le i<j\le n$, every element of $\G_j$ is contained in exactly one element of $\G_i$.
\item For any $s_1,\ldots,s_n\in\{-1,1\}$,
$$\la(\{x\in \Omega_n\,;\,T_j(x)=s_j\text{ for }j=1,\ldots,n\})=\dfrac{\la(\Omega_n)}{2^n}\,\cdot$$
\end{enumerate}
Observe that (ii) automatically implies that $\{x\in \Omega_n\,;\,T_j(x)=s_j\text{ for }j=1,\ldots,n\}$
is the union of elements of $\mathcal G_n$. Now let
$$\G_{n+1}=\{J\in\F_{n+1}\,;\,\exists\,I\in\G_n\text{ such that }J\subset I\}\ \ \text{ and }\ \
\Omega_{n+1}=\bigcup_{J\in\G_{n+1}}J.$$
For each $I\in\G_n$, the number of elements of $\G_{n+1}$ entirely contained in $I$ is between
$q_{n+1}/q_n-2$ and $q_{n+1}/q_n$. We may assume, removing at most two of these plateaux, that
\begin{equation}\label{rw - eq 1}
\#\left\{J\in\mathcal G_{n+1}^+\,;\,J\subset I\right\}=\#\left\{J\in\mathcal G_{n+1}^-\,;\,J\subset I\right\}
\end{equation}
and it is independent of $I$. (i) and (ii) are satisfied by definition. For (iii), fix
$s_1,\ldots,s_n\in\{-1,1\}$ and let $\G\subset\G_n$ such that
$$\{x\in \Omega_n\,;\,T_j(x)=s_j\text{ for }j=1,\ldots,n\}=\bigcup_{I\in\G}I.$$
Then
$$\{x\in \Omega_{n+1}\,;\,T_j(x)=s_j\text{ for }j=1,\ldots,n\text{ and }T_{n+1}(x)=1\}=
\bigcup_{I\in\mathcal G}\bigcup_{J\in\mathcal G_{n+1}^+\atop{J\subset I}}J$$
has Lebesgue measure equal to
$$\#\G\cdot\#\{J\in\mathcal G_{n+1}^+\,;\,J\subset I\}\cdot\dfrac{1}{2q_{n+1}}\,,$$
which is, by (\ref{rw - eq 1}), independent of $s_1,\ldots,s_n$. Doing the same when $T_{n+1}(x)=-1$,
(iii) is established.

The same argument applies to prove that, for $m\ge n$,
$$\la(\{x\in\Omega_m\,;\,T_j(x)=s_j\text{ for }j=1,\ldots,n\})=\dfrac{\la(\Omega_m)}{2^n}$$
and so, letting $m\rightarrow\infty$,
$$\la(\{x\in\Omega^\infty\,;\,T_j(x)=s_j\text{ for }j=1,\ldots,n\})=\dfrac{\la(\Omega^\infty)}{2^n}\,,$$
proving that the $(T_n)$ are independent in $\Omega^\infty$.

Now we estimate $\la(\Omega^\infty)$. By construction, inside each $I\in\mathcal G_n$ at most 4 elements
of $\F_{n+1}$ are removed and so
$$\la\left(\bigcup_{J\in\mathcal G_{n+1}\atop{J\subset I}}J\right)\ge\la(I)-4\cdot\dfrac{1}{2q_{n+1}}\,\cdot$$
Summing this up in $I$ yields
$$\la(\Omega_{n+1})\ge\la(\Omega_n)-\#\mathcal G_n\cdot\dfrac{2}{q_{n+1}}\ge\la(\Omega_n)-\dfrac{4q_n}{q_{n+1}}$$
and then
$$\la(\Omega^\infty)\ge 1-4\sum_{n\ge 1}\dfrac{q_n}{q_{n+1}}\,\cdot$$
If, instead of beginning the construction in step 1 we start in step $n_0$, the limit set $\Omega_{n_0}^\infty$
has Lebesgue measure at least $1-4\sum_{n\ge n_0}q_n/q_{n+1}$. By construction, the map that sends
$(x,y)\in\Omega_{n_0}^\infty\times\Omega_{n_0}^\infty$ to the process $(m_n(x)-m_n(y))_{n\ge n_0}$ is a random
walk in $\Z$ with finite support and zero mean, and thus it visits $0$ infinitely often almost surely
(see \S 4.2 of \cite{Dur}).

Since $\bigcup_{n_0\ge 1}\Omega_{n_0}^\infty$ has full Lebesgue measure, it follows that $(m_n)$ has the
level-crossing property.
\end{proof}

\section{A fact on continued fractions}\label{appendix continued fractions}

Remember the definition of Subsection
\ref{sub continued fractions}: $\al\in\R$ is {\it divisible} if it has a sequence $(q_{n_j})$ of
continuants satisfying
$$2q_{n_j}\text{ divides }q_{n_{j+1}}\ \ \text{ and }\ \ \lim_{j\rightarrow\infty} q_{n_j}\|q_{n_j}\al\|=0.$$
In this appendix, we want to prove that

\begin{proposition}\label{proposition cf 1}
Lebesgue almost every $\al\in\R$ is divisible.
\end{proposition}

To prove it, we first collect an auxiliary lemma and identify a mechanism to guarantee the divisibility
property. Once this is done, Proposition \ref{proposition cf 1} will follow. We acknowledge Carlos Gustavo
Moreira for communicating us this proof.

For positive integers $a_1,\ldots,a_n$, we recall the {\it continuant} $K(a_1,\ldots,a_n)$ denotes the
denominator of the rational number
$$[0;a_1,\ldots,a_n]\ =\ \dfrac{1}{a_1+\dfrac{1}{a_2+\dfrac{\ \ \ 1\ \ \ }{\ddots+\dfrac{1}{a_n}}}}\,\cdot$$

\begin{lemma}\label{appendix - lemma 1}
Let $n\ge 3$, $a_1,a_2,\ldots,a_{n-1}$ and $q$ be positive integers. Then there exist integers
$a,b$ such that if
$$
\left\{
\begin{array}{rcl}
a_n    &\equiv&a\pmod q\\
a_{n+1}&\equiv&b\pmod q
\end{array}\right.
$$
then $q$ divides $K(a_1,a_2,\ldots,a_n,a_{n+1})$.
\end{lemma}

\begin{proof}
Let $a$ be the product of the primes that divide $q$ and do not divide neither of the continuants $K(a_1,a_2,\ldots,a_{n-2})$, $K(a_1,a_2,\ldots,a_{n-1})$. If $a_n\equiv a\pmod q$, then
$$K(a_1,a_2,\ldots,a_n)= a\cdot K(a_1,a_2,\ldots,a_{n-1})+K(a_1,a_2,\ldots,a_{n-2})$$
and $q$ are coprime. This guarantees that, as $b$ varies modulo $q$, the number
$$K(a_1,a_2,\ldots,a_n,a_{n+1})=b\cdot K(a_1,a_2,\ldots,a_n)+K(a_1,a_2,\ldots,a_{n-1})$$
runs over all residues modulo $q$ and so, for one of these classes, it is divisible by $q$.
\end{proof}

The auxiliary lemma concerns the following elementary facts about continued fractions and continuants.

\begin{lemma}\label{appendix - lemma 3}
Let $\alpha=[a_0;a_1,a_2,\ldots]$ be an irrational number.
\begin{itemize}
\item[(a)] If $(q_n)$ is the sequence of continuants of $\al$, then
$$\dfrac{1}{a_{n+1}+2}<q_n\|q_n\al\|<\dfrac{1}{a_{n+1}}\,\cdot$$
\item[(b)] The probability that $a_{n+1}=k$, given that $a_1=k_1,\ldots,a_n=k_n$,
is between $\frac{1}{(k+1)(k+2)}$ and $\frac{2}{k(k+1)}$.
\item[(c)] The probability that $a_{n+1}\ge k$, given that $a_1=k_1,\ldots,a_n=k_n$,
is between $\frac{1}{k+1}$ and $\frac{2}{k}$.
\end{itemize}
\end{lemma}

\begin{proof}
(a) is a well-known fact and can be checked in any introductory text of continued fractions.
Let's prove (b). Once $a_1,\ldots,a_n$ are fixed, the number $\al=[0;a_1,\ldots,a_n,\al_{n+1}]$
belongs to the interval with endpoints $\frac{p_n}{q_n}$ and $\frac{p_n+p_{n-1}}{q_n+q_{n-1}}$.
In these conditions, $a_{n+1}=k$ if and only if $\al$ belongs to the interval of endpoints
$\frac{kp_n+p_{n-1}}{kq_n+q_{n-1}}$ and $\frac{(k+1)p_n+p_{n-1}}{(k+1)q_n+q_{n-1}}$.
Using the relation $|p_n q_{n-1}-p_{n-1}q_n|=1$, it follows that the ratio of the lengths of
these two intervals is equal to
$$\dfrac{q_n(q_n+q_{n-1})}{[kq_n+q_{n-1}][(k+1)q_n+q_{n-1}]}=
\dfrac{1+\frac{q_{n-1}}{q_n}}{\left(k+\frac{q_{n-1}}{q_n}\right)\left(k+1+\frac{q_{n-1}}{q_n}\right)}\ ,$$
which belongs to $\left[\frac{1}{(k+1)(k+2)}\,,\frac{2}{k(k+1)}\right]$. This establishes (b).
To prove (c), just observe that
$$\sum_{j\ge k}\dfrac{1}{(j+1)(j+2)}=\dfrac{1}{k+1}\ \text{ and }\
\sum_{j\ge k}\dfrac{2}{j(j+1)}=\dfrac{2}{k}\,\cdot$$
\end{proof}

\begin{proof}[Proof of Proposition \ref{proposition cf 1}]

For each positive integer $q$, let $D_q$ be the set of $\al\in\R$ for which there are
infinitely many $n\in\N$ such that $q$ divides $q_n$  and $a_{n+1}\ge n$.\\

\noindent {\bf Claim.} $D_q$ has full Lebesgue measure.\\

We prove this via the auxiliary lemmas. Assume that $a_1,a_2,\dots,a_{3k-1}$ are given. By Lemma
\ref{appendix - lemma 1}, there are $a,b\in\{1,\ldots,q\}$ such that
$K(a_1,a_2,\ldots,a_{3k-1},a,b)$ is divisible by $q$. By Lemma \ref{appendix - lemma 3}, the
probability that $a_{3k}=a$, $a_{3k+1}=b$ and $a_{3k+2}\ge 3k+1$ is at least $\frac{1}{(q+1)^2(q+2)^2(3k+1)}$.
Thus, given $k_0\ge 1$, the probability that, for each $k\ge k_0$, either $q_n$ is not a multiple of $q$
or $a_{n+1}<n$, is at most
$$\prod_{k \ge k_0}\left(1-\dfrac{1}{(q+1)^2(q+2)^2(3k+1)}\right)=0.$$
This proves the claim.\\

To conclude the proof of the proposition, consider the intersection $\bigcap_{q\ge 1}D_q$,
which by the above claim has full Lebesgue measure. Each $\al\in \bigcap_{q\ge 1}D_q$
is divisible. Indeed, one can inductively construct a sequence $(n_j)$ such that $2q_{n_j}$ divides $q_{n_{j+1}}$ and
$a_{n_j+1}\ge n_j$. Observing that, by Lemma \ref{appendix - lemma 3},
$$\lim_{j\rightarrow\infty}q_{n_j}\|q_{n_j}\al\|\le \lim_{j\rightarrow\infty}\dfrac{1}{a_{n_j+1}}=0,$$
the proof is complete.
\end{proof}

\begin{remark}
The above argument, together with the fact that, for Lebesgue almost every $\al\in\R$, $(q_n)$
grows at most (and at least) exponentially fast, can be used to show that Lebesgue almost every
$\al\in\R$ has a sequence of continuants $(q_{n_j})$ such that $2q_{n_j}$ divides $q_{n_{j+1}}$ and
$$q_{n_j}\|q_{n_j}\|<\dfrac{1}{\log q_{n_j}}\,\cdot$$
\end{remark}

\bibliographystyle{plain}

\begin{thebibliography}{99}

\bibitem{Aa2}
{\bf J. Aaronson},
\newblock{\it Rational ergodicity and a metric invariant for Markov shifts},
\newblock Israel J. Math. {\bf 27} (1977), no. 2, 93--123.

\bibitem{Aa3}
{\bf J. Aaronson},
\newblock{\it Ergodic theory for inner functions of the upper half plane},
\newblock Ann. Inst. H. Poincaré Sect. B (N.S.) {\bf 14} (1978), no. 3, 233--253.

\bibitem{Aa4}
{\bf J. Aaronson},
\newblock{\it An introduction to infinite ergodic theory},
\newblock Mathematical surveys and monographs {\bf 50} (1997), American Mathematical Society, Providence, RI.

\bibitem{ADF}
{\bf J. Aaronson, M. Denker and A. Fisher},
\newblock{\it Second order ergodic theorems for ergodic transformations of infinite measure spaces},
\newblock Proc. Amer. Math. Soc. {\bf 114} (1992), no. 1, 115--127.

\bibitem{AaK}
{\bf J. Aaronson and M. Keane},
\newblock{\it The visits to zero of some deterministic random walks},
\newblock Proc. London Math. Soc. {\bf 44} (1982), no. 3, 535--553.

\bibitem{ALMN92}
{\bf J. Aaronson, M. Leman\'czyk, C. Mauduit, and H. Nakada},
\newblock{\it Koksma's inequality and group extensions of Kronecker transformations},
\newblock Algorithms, fractals, and dynamics (Okayama/Kyoto, 1992), 27--50, Plenum, New York (1995).

\bibitem{AaS}
{\bf J. Aaronson and D. Sullivan},
\newblock{\it Rational ergodicity of geodesic flows},
\newblock Ergodic Theory \& Dynamical Systems {\bf 4} (1984), no. 2, 165--178.

\bibitem{AnKa}
{\bf D. Anosov and A. Katok},
\newblock{\it New examples in smooth ergodic theory. Ergodic diffeomorphisms},
\newblock Transactions of the Moscow Mathematical Society {\bf 23} (1970), 1--35.

\bibitem{BM92}
{\bf L. Baggett and K. Merrill},
\newblock{\it Smooth cocycles for an irrational rotation},
\newblock Israel J. Math. {\bf 79} (1992), no. 2-3, 281--288.

\bibitem{Co80}
{\bf J.-P. Conze},
\newblock{\it Ergodicit\'e d'une transformation cylindrique (French)},
\newblock Bull. Soc. Math. France {\bf 108} (1980), no. 4, 441--456.

\bibitem{CK76}
{\bf J.-P. Conze and M. Keane},
\newblock{\it Ergodicit\'e d'un flot cylindrique (French)},
\newblock S\'eminaire de Probabilit\'es I, Exp. No. 5, 7 pp. D\'ept. Math. Informat., Univ.
Rennes, Rennes (1976).

\bibitem{Dur}
{\bf R. Durrett},
\newblock{\it Probability: theory and examples}, The Wadsworth \& Brooks/Cole Statistics/Probability Series,
\newblock Pacific Grove (1991).

\bibitem{FL06}
{\bf B. Fayad and M. Lema\'nczyk},
\newblock{\it On the ergodicity of cylindrical transformations given by the logarithm},
\newblock Mosc. Math. J.{\bf 6} (2006), no. 4, 657–-672, 771--772.

\bibitem{Fr00}
{\bf K. Fr\c aczek},
\newblock{\it On ergodicity of some cylinder flows},
\newblock Fund. Math. {\bf 163} (2000), no. 2, 117–-130.

\bibitem{Ka03}
{\bf A. Katok},
\newblock{\it Combinatorial constructions in ergodic theory and dynamics},
\newblock University Lecture Series {\bf 30}. American Mathematical Society, Providence ,RI (2003).

\bibitem{Kr74}
{\bf A. Krygin},
\newblock{\it Examples of ergodic cylindrical cascades (Russian)},
\newblock Mat. Zametki {\bf 16} (1974), 981--991.

\bibitem{LS}
{\bf F. Ledrappier and O. Sarig},
\newblock{\it Unique ergodicity for non-uniquely ergodic horocycle flows},
\newblock Discrete Contin. Dyn. Syst. {\bf 16} (2006), no. 2,  411--433.

\bibitem{Or83}
{\bf I. Oren},
\newblock{\it Ergodicity of cylinder flows arising from irregularities of distribution},
\newblock Israel J. Math. {\bf 44} (1983), no. 2, 127--138.

\bibitem{Pa91}
{\bf D. Pask},
\newblock{\it Ergodicity of certain cylinder flows},
\newblock Israel J. Math. {\bf 76} (1991), no. 1-2, 129--152.

\bibitem{Sa10}
{\bf O. Sarig},
\newblock{\it Unique ergodicity for infinite measures},
\newblock Proc. Inter. Congress Math., Hyderabad (2010).

\bibitem{Sc77}
{\bf K. Schmidt},
\newblock{\it Cocycles on ergodic transformation groups},
\newblock Macmillan Lectures in Mathematics, vol. 1, Macmillan Company of India (1977).

\bibitem{Sc78}
{\bf K. Schmidt},
\newblock{\it A cylinder flow arising from irregularity of distribution},
\newblock Compositio Math. {\bf 36} (1978), no. 3, 225--232.

\end{thebibliography}

\end{document}